\documentclass[draft,reqno]{amsproc}
\usepackage{amssymb}
\usepackage{graphicx}
\usepackage{euscript}
\usepackage{mathrsfs} 
\usepackage{units}   
\usepackage{color}

 \makeatletter
\@namedef{subjclassname@2010}{%
\textup{2010} Mathematics Subject Classification}
\makeatother

\numberwithin{equation}{section}

\newtheorem{thm}{Theorem}[section]
\newtheorem{cor}[thm]{Corollary}
\newtheorem{lem}[thm]{Lemma}
\newtheorem{pro}[thm]{Proposition}

\newtheorem*{thm*}{Theorem}

\theoremstyle{remark}
\newtheorem{rem}[thm]{Remark}

\newtheorem{opq}[thm]{Problem}

\theoremstyle{definition}
\newtheorem{exa}[thm]{Example}

\DeclareMathOperator{\E}{e}
\DeclareMathOperator{\koo}{{\mathsf{root}}}

\DeclareMathOperator{\paa}{{\mathsf{par}}}
\DeclareMathOperator{\sumo}{{\sum\nolimits^{\smalloplus}}}

\newcommand*{\borel}[1]{{\mathfrak B}(#1)}
\newcommand*{\cbb}{\mathbb C}

\newcommand*{\coms}[1]{{\mathfrak C_{\mathrm{s}}(#1)}}

\newcommand*{\dz}[1]{{\EuScript D}(#1)}

\newcommand*{\dzn}[1]{{\EuScript D}^\infty(#1)}
\newcommand*{\D}{\mathrm{d}}
\newcommand*{\ee}{\mathcal E}

\newcommand*{\ff}{\mathcal F}
\newcommand*{\Ge}{\geqslant}
\newcommand*{\hh}{\mathcal H}
\newcommand*{\hinf}{{\EuScript H}_{\infty}}
\newcommand*{\I}{{\mathrm i}}
\newcommand*{\is}[2]{\langle#1,#2\rangle}
\newcommand*{\jd}[1]{\EuScript N(#1)}

\newcommand*{\lambdab}{{\boldsymbol\lambda}}

\newcommand*{\Le}{\leqslant}
\newcommand*{\mm}{\mathcal M}

\newcommand*{\nbb}{\mathbb N}
\newcommand*{\ogr}[1]{\boldsymbol B(#1)}
\newcommand*{\ob}[1]{{\mathcal R}(#1)}
\newcommand*{\pa}[1]{\paa(#1)}

\newcommand*{\rbb}{\mathbb R}
\newcommand*{\slam}{S_{\boldsymbol \lambda}}
\newcommand*{\smalloplus}{\raise0pt\hbox{$\scriptscriptstyle \oplus$}}

\newcommand*{\tcal}{{\mathscr T}}

\newcommand*{\zbb}{\mathbb Z}

\begin{document}
   \title[Unbounded quasinormal operators revisited]
{Unbounded quasinormal operators revisited}
   \author[Z.\ J.\ Jab{\l}o\'nski]{Zenon Jan Jab{\l}o\'nski}
\address{Instytut Matematyki, Uniwersytet Jagiello\'nski,
ul.\ \L ojasiewicza 6, PL-30348 Kra\-k\'ow, Poland}
   \email{Zenon.Jablonski@im.uj.edu.pl}
   \author[I.\ B.\ Jung]{Il Bong Jung}
   \address{Department of Mathematics, Kyungpook National
University, Daegu 702-701, Korea}
   \email{ibjung@knu.ac.kr}
   \author[J.\ Stochel]{Jan Stochel}
\address{Instytut Matematyki, Uniwersytet Jagiello\'nski,
ul.\ \L ojasiewicza 6, PL-30348 Kra\-k\'ow, Poland}
   \email{Jan.Stochel@im.uj.edu.pl}
   \thanks{Research of the second
author was supported by the National Research
Foundation of Korea(NRF) grant funded by the Korea
government(MSIP) (No.2009-0083521).}
   \begin{abstract}
Various characterizations of unbounded closed densely
defined operators commuting with the spectral measures
of their moduli are established. In particular,
Kaufman's definition of an unbounded quasinormal
operator is shown to coincide with that given by the
third-named author and Szafraniec. Examples
demonstrating the sharpness of results are
constructed.
   \end{abstract}
   \maketitle
   \section{Introduction}
The class of bounded quasinormal operators was
introduced by A. Brown in \cite{bro}. Two different
definitions of unbounded quasinormal operators
appeared independently in \cite{kau} and (a few years
later) in \cite{StSz1}. In the present paper we show
that both of these definitions coincide (cf.\ Theorem
\ref{thS1}). We also discuss the question of whether
the equality in Kaufman's definition of quasinormality
can be replaced by inclusion. It is shown that the
answer is in the affirmative if the inclusion is
properly chosen (cf.\ Theorem \ref{thS1} and Remark
\ref{sq0}). Next, we characterize quasinormality of
unbounded operators in terms of the truncated operator
Stieltjes moment problem (cf.\ Theorem
\ref{main}(iv)). This part of the paper is inspired by
a result of Embry which characterizes quasinormality
of bounded operators by means of the operator
Stieltjes moment problem (cf.\ \cite[page 63]{Em}).
Yet another characterization of quasinormality of
unbounded operators is given in Theorem \ref{main}(v).
It states that a closed densely defined operator $C$
is quasinormal if and only if the equality
$C^{*n}C^{n} = (C^*C)^n$ holds for $n=2,3$. In the
case of bounded operators, this characterization has
been known for specialists working in this area since
late '80s (recently it has been published in
\cite[Proposition 13]{Jib}; unfortunately, this paper
contains several errors). The proof of Theorem
\ref{main} is essentially more advanced and requires
the use of the technique of bounded vectors. That the
single equality $C^{*n}C^{n} = (C^*C)^n$ does not
imply the quasinormality of $C$ is elucidated by
Examples \ref{prz1}, \ref{prz2} and \ref{prz3}. The
first example is related to Toeplitz operators on the
Hardy space $H^2$, while two others are linked to
weighted shifts on a directed tree with one branching
vertex of valency $2$. Example \ref{prz2} enables us
to construct a bounded injective non-quasinormal
composition operator $C$ on an $L^2$ space over a
$\sigma$-finite measure space such that $(C^*C)^2 =
C^{*2}C^2$ (cf.\ Remark \ref{construct}). Example
\ref{prz3} shows how to separate (with respect to $n$)
the classes of operators $C$ satisfying the equality
$C^{*n}C^{n} = (C^*C)^n$. In Example \ref{prz4} we
construct a quasinormal operator $C$ such that $C^{*n}
\subsetneq (C^n)^*$ for every $n \Ge 2$. In Section
\ref{sect} we show that closed densely defined
operators $C$ which satisfy the equality $C^{*2}C^{2}
= (C^*C)^2$ and have closed squares, are paranormal
(cf.\ Theorem \ref{parapara}). As a consequence,
operators constructed in Examples \ref{prz1},
\ref{prz2} and \ref{prz3} are instances of bounded
injective paranormal operators which are not
hyponormal (see the first paragraph of Section
\ref{s4}). The paper concludes with an appendix
discussing orthogonal sums of unbounded operators.
Moreover, two open problems are formulated (cf.\
Problems \ref{opq0} and \ref{opq1}).

   In what follows $\nbb$, $\zbb_+$ and $\rbb_+$ stand
for the sets of positive integers, nonnegative
integers and nonnegative real numbers respectively.
Denote by $\borel{X}$ the $\sigma$-algebra of all
Borel subsets of a topological space $X$. Let $A$ be
an operator in a complex Hilbert space $\hh$ (all
operators considered in this paper are linear). Denote
by $\dz{A}$, $\jd{A}$, $\ob{A}$, $A^*$, $\bar A$ and
$\sigma(A)$ the domain, the kernel, the range, the
adjoint, the closure and the spectrum of $A$ (in case
they exist) respectively. Set $\dzn{A} =
\bigcap_{n=0}^\infty\dz{A^n}$. In what follows, we
write $C^{*n}$ in place of $(C^*)^n$. A vector
subspace $\ee$ of $\dz{A}$ is said to be a {\em core}
for $A$ if $\ee$ is dense in $\dz{A}$ with respect to
the graph norm of $A$. If $A$ is closed, then $\ee$ is
a core for $A$ if and only if $A=\overline{A|_{\ee}}$,
where $A|_{\ee}$ stands for the restriction of $A$ to
$\ee$. A closed densely defined operator $A$ has a
(unique) polar decomposition $A = U|A|$, where $U$ is
a partial isometry on $\hh$ such that $\jd{U} =
\jd{A}$ and $|A|$ is the square root of $A^*A$ (cf.\
\cite[Theorem 7.20]{weid}). If $A$ and $B$ are two
operators in $\hh$ such that the graph of $A$ is
contained in the graph of $B$, then we write $A
\subseteq B$ or $B \supseteq A$. In what follows,
$\ogr \hh$ stands for the $C^*$-algebra of all bounded
operators $A$ in $\hh$ such that $\dz{A}=\hh$. Denote
by $I$ the identity operator on $\hh$. Recall that a
closed densely defined operator $N$ in $\hh$ is {\em
normal} if $N^*N = NN^*$. If $N$ is a normal operator,
then its spectral measure, denoted here by $E_N$ and
usually defined on $\borel{\sigma(N)}$, can be thought
of as a spectral Borel measure on any fixed closed
subset of $\cbb$ containing $\sigma(N)$. In any case,
the closed support of such measure coincides with
$\sigma(N)$. For this and other facts concerning
unbounded operators we refer the reader to
\cite{b-s,weid}.
   \section{Commutativity}
In this section we discuss the question of
commutativity of a bounded normal operator with an
unbounded closed operator. Though this is a more
general approach than we need in this paper, the
question itself seems to be of independent interest.

Given $\ff \subseteq \ogr{\hh}$, we write $\ff^\prime
= \{T \in \ogr{\hh}\colon \forall A \in \ff, \,
TA=AT\}$ and $\ff^{\prime\prime} =
(\ff^{\prime})^\prime$. If $A$ is an operator in
$\hh$, then we set
   \begin{align*}
\coms{A} = \{T \in \ogr{\hh}\colon TA \subseteq AT
\text{ and } T^*A \subseteq AT^*\}.
   \end{align*}
Let $N$ be a normal operator in $\hh$. For a complex
Borel function $\phi$ on $\sigma(N)$, we abbreviate
$\int_{\sigma(N)} \phi \, \D E_N$ to $\phi(N)$ (cf.\
\cite{b-s}). If $A$ is an operator in $\hh$ such that
$E_N(\varDelta) A \subseteq A E_N(\varDelta)$ for
every $\varDelta \in \borel{\sigma(N)}$, then we write
$E_N A \subseteq A E_N$.
   \begin{lem} \label{lemS1}
Let $A$ be a closed operator in $\hh$ and $N$ be a
normal operator in $\hh$. Consider the following three
conditions{\em :}
   \begin{enumerate}
   \item[(i)] $NA \subseteq AN$
and $N^*A \subseteq AN^*$,
   \item[(ii)] $E_N A \subseteq A E_N$,
   \item[(iii)] $\phi(N) A \subseteq A
\phi(N)$ for every bounded complex Borel function
$\phi$ on $\sigma(N)$.
   \end{enumerate}
Then the conditions {\em (ii)} and {\em (iii)} are
equivalent. Moreover, if $N$ is bounded, then the
conditions {\em (i)}, {\em (ii)} and {\em (iii)} are
equivalent.
   \end{lem}
   \begin{proof}
Assume (ii) holds. For $h\in \hh$, we set
$\mu_h=\is{E_N(\cdot)h}{h}$. Take $f\in \dz{A}$.
Suppose $\phi$ is a bounded complex Borel function on
$\sigma(N)$. Then clearly $\phi \in L^2(\mu_f +
\mu_{Af})$. By \cite[Theorem 3.13]{Rud}, there exists
a sequence $\{s_n\}_{n=1}^\infty$ of complex Borel
simple functions on $\sigma(N)$ such that $\lim_{n\to
\infty} \int_{\sigma(N)} |\phi - s_n|^2 \D (\mu_f +
\mu_{Af}) = 0$. Hence, we have
   \allowdisplaybreaks
   \begin{align*}
\|\phi(N) f - s_n(N) f\|^2 & = \int_{\sigma(N)} |\phi
- s_n|^2 \, \D \mu_f \to 0 \text{ as } n \to \infty,
   \\
\|\phi(N) A f - s_n(N) A f\|^2 & = \int_{\sigma(N)}
|\phi - s_n|^2 \, \D \mu_{Af} \to 0 \text{ as } n \to
\infty.
   \end{align*}
Since, by our assumption, $s_n(N) A \subseteq A
s_n(N)$ for all $n\Ge 1$, we deduce that $s_n(N) f \to
\phi(N) f$ and $A (s_n(N) f) \to \phi(N) A f$ as $n\to
\infty$. This and the closedness of $A$ imply (iii).
The implication (iii)$\Rightarrow$(ii) is obvious.

Now suppose $N$ is bounded and (i) holds. Since $A$ is
closed, $\coms{A}$ is a von Neumann algebra with unit
$I$. By assumption $N \in \coms{A}$ and thus
$\coms{A}$ contains the von Neumann algebra
$\mathcal{W}^*(I,N)$ generated by $I$ and $N$. By von
Neumann's double commutant theorem\footnote{\;In fact,
by the Fuglede theorem $\{N,N^*\}^{\prime\prime} =
\{N\}^{\prime\prime}$.}, $\mathcal{W}^*(I,N) =
\{N,N^*\}^{\prime\prime}$. Since
$E_N(\borel{\sigma(N)}) \subseteq
\{N,N^*\}^{\prime\prime}$, we get (ii). The
implication (iii)$\Rightarrow$(i) is obvious. This
completes the proof.
   \end{proof}
   \begin{rem}
Regarding Lemma \ref{lemS1}, we note that by
approximating $\phi$ by polynomials in two complex
variables $z$ and $\bar z$ and arguing as in the proof
of (ii)$\Rightarrow$(iii), we can show that
(i)$\Rightarrow$(ii). It is worth pointing out that if
$A$ is closed and densely defined and $\sigma(A) \neq
\cbb$, then Lemma \ref{lemS1} remains true if one
drops the assumption that $N^*A \subseteq AN^*$.
Indeed, the inclusion $NA \subseteq AN$ implies that
$(\lambda - A)^{-1} N \subseteq N (\lambda - A)^{-1}$
for any fixed $\lambda \in \cbb\setminus \sigma(A)$,
and thus by the Fuglede theorem, for each bounded
complex Borel function $\phi$ on $\sigma(N)$,
$(\lambda - A)^{-1} \phi(N) = \phi(N) (\lambda -
A)^{-1}$, which in turn implies that $\phi(N) A
\subseteq A \phi(N)$ (in particular, $N^*A \subseteq
AN^*$).
   \end{rem}
The above remark suggests the following problem which
is related to the Fuglede theorem.
   \begin{opq}  \label{opq0}
Let $A$ be a closed operator in $\hh$ and $N$ be a
bounded normal operator in $\hh$ such that $NA
\subseteq AN$. Is it true that $N^*A \subseteq AN^*$?
   \end{opq}
Now we consider the issue of commutativity of an
unbounded operator with a spectral measure of a
positive selfadjoint operator.
   \begin{lem} \label{lemS3}
Let $A$ be a closed operator in $\hh$ and $R$ be a
positive selfadjoint operator in $\hh$. Then the
following conditions are equivalent{\em :}
   \begin{enumerate}
   \item[(i)] $E_R A \subseteq A E_R$,
   \item[(ii)] $(I + R)^{-1} A \subseteq A (I +
R)^{-1}$,
   \item[(iii)] $\phi(R) A \subseteq A \phi(R)$
for every bounded complex Borel function $\phi$ on
$\sigma(R)$.
   \end{enumerate}
   \end{lem}
   \begin{proof}
(i)$\Rightarrow$(ii) Apply Lemma \ref{lemS1} with
$\phi(x) = \frac{1}{1+x}$.

(ii)$\Rightarrow$(i) Set $N=(I + R)^{-1}$ and note
that $N\in \ogr{\hh}$. Applying Lemma \ref{lemS1}, we
see that $E_{N} A \subseteq A E_{N}$. Since $R$ and
$N$ are positive and selfadjoint, we may (and do)
regard $E_R$ and $E_{N}$ as spectral Borel measures on
$\rbb_+$. Let $\phi\colon [0,\infty) \to (0,1]$ be the
homeomorphism given by $\phi(x)= \frac{1}{1+x}$ for $x
\in [0,\infty)$. By \cite[Theorem 5.4.10]{b-s}, we
have
   \begin{align*}
N = \int_{[0,\infty)} \phi(x) \, E_R (\D x) =
\int_{(0,1]} t \, E_R \circ \phi^{-1}(\D t),
   \end{align*}
where $E_R \circ \phi^{-1}$ is the spectral measure
given by $(E_R \circ \phi^{-1})(\varDelta) =
E_R(\phi^{-1}(\varDelta))$ for $\varDelta \in
\borel{(0,1]}$. By the uniqueness assertion in the
spectral theorem, this implies that $E_N (\varDelta) =
E_R(\phi^{-1}(\varDelta \cap (0,1]))$ for all
$\varDelta \in \borel{\rbb_+}$, and thus
   \begin{align*}
E_R(\varDelta) = E_R(\phi^{-1} (\phi(\varDelta))) =
E_N (\phi(\varDelta)), \quad \varDelta \in
\borel{\rbb_+}.
   \end{align*}
This and the fact that $E_{N}(\varDelta) A \subseteq A
E_{N}(\varDelta)$ for all $\varDelta \in
\borel{\rbb_+}$ yield (i).

(i)$\Leftrightarrow$(iii) Apply Lemma \ref{lemS1}.
   \end{proof}
   \section{Quasinormality revisited}
Following \cite{StSz1} (see \cite[Lemma 4.1]{bro} for
the bounded case) we say that a closed densely defined
operator $C$ in $\hh$ is {\em quasinormal} if $C$
commutes with $E_{|C|}$, i.e., $E_{|C|}C \subseteq C
E_{|C|}$. By \cite[Proposition 1]{StSz1}, a closed
densely defined operator $C$ in $\hh$ is quasinormal
if and only if $U |C| \subseteq |C|U$, where $C=U|C|$
is the polar decomposition of $C$. It is well-known
that quasinormal operators are always subnormal and
that the reverse implication does not hold in general.
For more information on quasinormal operators we refer
the reader to \cite{bro,con}, the bounded case, and to
\cite{StSz1,maj}, the unbounded one. Our aim in this
section is to show that the above definition of
quasinormality coincide with the one given by Kaufman
in \cite{kau} (see condition (ii) below). In fact, we
prove that the equality in Kaufman's definition can be
replaced by inclusion (see condition (i) below).
Recall that if $C$ is a closed densely defined
operator in $\hh$, then $C^*C$ is a positive
selfadjoint operator (cf.\ \cite[Section 4.5]{b-s}).
   \begin{thm} \label{thS1}
Let $C$ be a closed densely defined operator in $\hh$.
Then the following conditions are equivalent{\em :}
   \begin{enumerate}
   \item[(i)] $CC^*C \subseteq C^*CC$,
   \item[(ii)] $CC^*C = C^*CC$,
   \item[(iii)] $(I+C^*C)^{-1} C \subseteq C
(I+C^*C)^{-1}$,
   \item[(iv)] $E_{|C|}C \subseteq CE_{|C|}$.
   \end{enumerate}
   \end{thm}
   \begin{proof}
(i)$\Rightarrow$(iii) It follows from (i) that $C(I +
C^*C) \subseteq (I + C^*C) C$. This implies that $(I +
C^*C)^{-1} C (I + C^*C) \subseteq C$, which yields
(iii).

(iii)$\Rightarrow$(iv) Apply \cite[Theorem
5.4.10]{b-s} and Lemma \ref{lemS3} with $\phi(x) =
\chi_{\varDelta}(\sqrt{x})$, where $\chi_{\varDelta}$
is the characteristic function of a set $\varDelta \in
\borel{\rbb_+}$.

(iv)$\Rightarrow$(ii) Let $C=U|C|$ be the polar
decomposition of $C$. By \cite[Proposition 1]{StSz1}
and \cite[Lemma 2.2]{ota2}, we have $U|C| = |C|U$,
which implies that
   \begin{align*}
CC^*C = U|C|^3 = |C|^2U|C| = C^*C C.
   \end{align*}

(ii)$\Rightarrow$(i) Evident.
   \end{proof}
   \begin{rem}  \label{sq0}
It is worth mentioning that the inclusion $C^*CC
\subseteq CC^*C$, which is opposite to the one in
Theorem \ref{thS1}(i), do not imply quasinormality. To
see this, take a nonzero closed densely defined
operator $C$ such that $\dz{C^2} = \{0\}$ (see
\cite{Nai,Cher} for the case of symmetric operators,
or \cite{Bu,j-j-s3} for the case of hyponormal
composition operators). Then $\dz{C^*CC} = \{0\}$ and
thus $C^*CC \subseteq CC^*C$. However, $C$ is not
quasinormal. Indeed, otherwise, by Theorem \ref{thS1},
$C^*CC = CC^*C = U|C|^3$, which implies that $C^*CC$
is densely defined, a contradiction.
   \end{rem}
To prove Theorem \ref{main}, which is one of the main
results of this paper, we need three preparatory
lemmas.
   \begin{lem} \label{lemS4}
If $C$ is a closed densely defined operator in $\hh$,
then $\hinf \subseteq \dzn{C^*C}$ and $(I+C^*C) \hinf
= \hinf$, where $\hinf = \bigcup_{n=1}^\infty
\ob{E_{|C|}([0,n])}$.
   \end{lem}
   \begin{proof}
Set $R=C^*C$ and $\hh_n=\ob{E_{|C|}([0,n])}$ for $n
\in \nbb$. Fix $n \in \nbb$. It is clear that $\hh_n
=\overline{\hh}_{n}$ and $\hh_n \subseteq \dz{R}$. By
\cite[Theorem 6.3.2]{b-s}, the space $\hh_n$ reduces
$R$, and thus $\hh_{n} \subseteq \dzn{R}$. By the
closed graph theorem (or via \cite[Chapter 5]{b-s}),
the operator $R_{n}:=R|_{\hh_n}$ is bounded. Since
$R_{n}$ is positive, we deduce that $(I + R_{n})^{-1}
\in \ogr{\hh_{n}}$ and consequently $(I+R)\hh_{n} =
\ob{I + R_{n}}=\hh_{n}$. This completes the proof.
   \end{proof}
   \begin{lem} \label{lemS4.5}
If $C$ is a closed densely defined operator in $\hh$
and $n\in \zbb_+$, then
   \begin{enumerate}
   \item[(i)] $(C^*C)^n \subseteq C^{*n}C^n$ if and only if
$(C^*C)^n = C^{*n}C^n$,
   \item[(ii)] $(C^*C)^n \subseteq (C^{n})^*C^n$ if and only if
$(C^*C)^n = (C^{n})^*C^n$ provided $C^n$ is densely
defined,
   \item[(iii)] if $(C^*C)^n = C^{*n}C^{n}$, then $(C^*C)^n =
(C^n)^* C^{n}$.
   \end{enumerate}
   \end{lem}
   \begin{proof}
(iii) Observe that the operator $(C^*C)^n$ is
selfadjoint (see e.g., \cite[Theorems 7.14 and
7.19]{weid}). Since $(C^*C)^n = C^{*n}C^{n}$, we see
that $C^n$ is densely defined and consequently
$(C^*C)^n \subseteq (C^n)^* C^{n}$. As selfadjoint
operators are maximal symmetric and $(C^n)^* C^{n}$ is
symmetric, we get $(C^*C)^n = (C^n)^* C^{n}$.

Similar argument can be used to prove (i) and (ii).
   \end{proof}
   \begin{lem} \label{lemS5}
If $C$ is a quasinormal operator in $\hh$, then
   \begin{align} \label{ZenEq}
(C^*C)^n = C^{*n}C^{n} = (C^n)^* C^{n}, \quad n\in
\zbb_+.
   \end{align}
   \end{lem}
   \begin{proof}
Let $C=U|C|$ be the polar decomposition of $C$. By
\cite[Lemma 2.2]{ota2}, we have $U|C| = |C|U$. First
we show that
   \begin{align} \label{Zen2}
|C|^k = U^* |C|^k U, \quad k\in \nbb.
   \end{align}
Indeed, it follows from $U|C| = |C|U$ that $U|C|^k =
|C|^kU$ for all $k\in \nbb$. Since $U^*U$ is the
orthogonal projection of $\hh$ onto
$\overline{\ob{|C|}}$, we get $|C|^k = U^*U |C|^k =
U^*|C|^kU$ for all $k\in \nbb$, which yields
\eqref{Zen2}.

Using induction on $n\Ge 1$, we will show that the
left-hand equality in \eqref{ZenEq} holds. Clearly it
is valid for $n=1$. Assuming it holds for a fixed
$n\Ge 1$ and noting that $(U|C|)^*=|C|U^*$, we get
   \allowdisplaybreaks
   \begin{align*}
C^{*(n+1)}C^{n+1} & = C^*(C^{*n}C^n)C = C^*(C^*C)^nC =
(U|C|)^* |C|^{2n} U|C|
   \\
& = |C| (U^* |C|^{2n} U) |C| \overset{\eqref{Zen2}}=
|C| |C|^{2n} |C| = (C^*C)^{n+1},
   \end{align*}
which proves our claim. Applying Lemma
\ref{lemS4.5}(iii) completes the proof.
   \end{proof}
Now we are ready to prove the afore-mentioned
characterization of quasinormal operators. In the case
of bounded operators the equivalences
(i)$\Leftrightarrow$(ii)$\Leftrightarrow$(iii) have
been proved by Embry in \cite[page 63]{Em}, and the
implication (v)$\Rightarrow$(i) can be found in
\cite{Jib}.
   \begin{thm} \label{main}
Let $C$ be a closed densely defined operator in $\hh$.
Then the following conditions are equivalent{\em :}
   \begin{enumerate}
   \item[(i)] $C$ is quasinormal,
   \item[(ii)] $C^{*n}C^{n} = (C^*C)^n$ for every $n \in
\zbb_+$,
   \item[(iii)] there exists a $($unique$)$ spectral
Borel measure $E$ on $\rbb_+$ such that
   \begin{align*}
C^{*n}C^{n} = \int_{\rbb_+} x^n E(\D x), \quad n\in
\zbb_+,
   \end{align*}
   \item[(iv)] there exists a $($unique$)$ spectral
Borel measure $E$ on $\rbb_+$ such that
   \begin{align*}
C^{*n}C^{n} = \int_{\rbb_+} x^n E(\D x), \quad n\in
\{1,2,3\},
   \end{align*}
   \item[(v)] $C^{*n}C^{n} = (C^*C)^n$ for  $n \in
\{2,3\}$.
   \end{enumerate}
   \end{thm}
Before proving Theorem \ref{main} we make a remark.
   \begin{rem}
By Lemma \ref{lemS4.5}(i), Theorem \ref{main} remains
valid if the equality in (ii) and (v) is replaced by
the inclusion ``$\supseteq$''. This is not the case
for ``$\subseteq$''. To see this, it is enough to
consider a nonzero closed densely defined operator $C$
such that $\dz{C^2} = \{0\}$ (consult Remark
\ref{sq0}).
   \end{rem}
   \begin{proof}[Proof of Theorem  \ref{main}]
Note that the uniqueness of $E$ in (iii) and (iv)
follows from the equality $C^*C = \int_{\rbb_+} x E(\D
x)$ and the spectral theorem. In both cases
$E=E_{C^*C}$.

(i)$\Rightarrow$(ii) Apply Lemma \ref{lemS5}.

(ii)$\Rightarrow$(iii) By the multiplicative property
of spectral integral (cf.\ \cite[Theorem 7.19]{weid}),
$E=E_{C^*C}$ meets our requirements.

(iii)$\Rightarrow$(iv) Evident.

(iv)$\Rightarrow$(v) This follows from the
multiplicative property of spectral integral and the
equality $E=E_{C^*C}$.

(v)$\Rightarrow$(i) By (v), we have
   \begin{align}   \label{IBJ2}
   \begin{aligned}
C^{*2}CC^*C^2 = C^*(C^*C)^2C = C^{*3}C^3 = (C^*C)^3,
   \\
C^*CC^{*2}C^2 = (C^*C)^3, \quad C^{*2}C^2 C^*C =
(C^*C)^3.
   \end{aligned}
   \end{align}
The penultimate equality yields $\dz{(C^*C)^3}
\subseteq \dz{C^*C^2 - CC^*C}$. Thus, we have
   \allowdisplaybreaks
   \begin{align*}
\|(C^*C^2 - CC^*C) f\|^2 = \is{C^{*2}CC^*C^2f}{f} & -
\is{C^*CC^{*2}C^2f}{f} - \is{C^{*2}C^2 C^*Cf}{f}
   \\
& + \is{(C^*C)^3f}{f} \overset{\eqref{IBJ2}} = 0
   \end{align*}
for every $f \in \dz{(C^*C)^3}$, which implies that
   \begin{align*}
C^*C^2 f = CC^*Cf, \quad f \in \dz{(C^*C)^3}.
   \end{align*}
Therefore this leads to
   \begin{align*}
C(I+C^*C)f = (I + C^*C)Cf, \quad f \in \dz{(C^*C)^3}.
   \end{align*}
Since $I+C^*C$ is injective and $(I+C^*C)^{-1} \in
\ogr{\hh}$, we deduce that
   \allowdisplaybreaks
   \begin{align*}
(I+C^*C)^{-1} C (I+C^*C)f &= Cf
   \\
&= C (I+C^*C)^{-1} (I + C^*C)f, \quad f \in
\dz{(C^*C)^3}.
   \end{align*}
Applying Lemma \ref{lemS4}, we obtain
   \begin{align} \label{Zen5}
(I+C^*C)^{-1} (C|_{\hinf}) = C (I+C^*C)^{-1}|_{\hinf}
\subseteq C (I+C^*C)^{-1}.
   \end{align}
It is well-known that $\hinf$ is a core for $|C|$
(see e.g., \cite[Proposition 4]{StSz3}). Since
the domains of $C$ and $|C|$ are equal and the
graph norms of $C$ and $|C|$ coincide, we deduce
that $\hinf$ is a core for $C$. This combined
with \eqref{Zen5} leads to
   \begin{align*}
(I+C^*C)^{-1} C = (I+C^*C)^{-1} \,
\overline{C|_{\hinf}} \subseteq
 C (I+C^*C)^{-1}.
   \end{align*}
By Theorem \ref{thS1} this completes the proof.
   \end{proof}
The following corollary can be viewed as a special
case of \cite[Proposition 9.1]{ota2}. Our proof is
different, less involved and much shorter.
   \begin{cor}
If $C$ is a quasinormal operator, then for every $n
\in \zbb_+$, $C^n$ is a quasinormal operator.
   \end{cor}
   \begin{proof}
Since quasinormal operators are subnormal, we infer
from \cite[Proposition 5.3.]{sto1} that $C^n$ is
closed. By \cite[Proposition 5]{StSz1}, $C^n$ is
densely defined. Fix $k\in \zbb_+$. It follows from
Theorem \ref{main} that $(C^n)^{*k} (C^n)^k \supseteq
C^{*nk} C^{nk} = (C^*C)^{nk}$. Since $(C^n)^{*k}
(C^n)^k$ is symmetric and $(C^*C)^{nk}$ is
selfadjoint, we get
   \begin{align} \label{no1}
(C^n)^{*k} (C^n)^k = (C^*C)^{nk}.
   \end{align}
Similarly, we see that $((C^n)^*C^n)^k \supseteq
(C^{*n}C^n)^k = (C^{*}C)^{nk}$. Since $((C^n)^*C^n)^k$
is symmetric (in fact selfadjoint) and $(C^*C)^{nk}$
is selfadjoint, we get $((C^n)^*C^n)^k =
(C^{*}C)^{nk}$. This, \eqref{no1} and Theorem
\ref{main} imply that $C^n$ is quasinormal.
   \end{proof}
In view of Lemma \ref{lemS5} and Theorem \ref{main},
the following problem arises.
   \begin{opq} \label{opq1}
Is it true that, if $C$ is a closed operator in $\hh$
such that $C^3$ is densely defined and\footnote{\;See
also Lemma \ref{lemS4.5}(ii).} $(C^*C)^n = (C^n)^*
C^{n}$ for $n\in \{2,3\}$, then $C$ is quasinormal?
   \end{opq}
   \section{\label{sect}More on $\boldsymbol{(C^*C)^2 = C^{*2}C^2}$}
   Regarding Theorem \ref{main}, it is tempting to ask
the question whether the condition (v) can be replaced
by the single equality $(C^*C)^2 = C^{*2}C^2$ without
affecting the conclusion. Though, the answer is in the
negative (cf.\ Section \ref{s4}), operators satisfying
this equality are often paranormal (see Theorem
\ref{parapara} and Corollary \ref{paran} below).
Recall that an operator $C$ in $\hh$ is said to be
{\em paranormal} if $\|Cf\|^2 \Le \|C^2 f\| \|f\|$ for
all $f \in \dz{C^2}$. We refer the reader to
\cite{Dan} for examples of unbounded paranormal
operators with pathological properties related to
closability (in particular, paranormal operators may
not be closable).
   \begin{thm}\label{parapara}
Let $C$ be a closed densely defined operator in $\hh$
such that
   \begin{align} \label{zal}
(C^*C)^2 = C^{*2}C^2.
   \end{align}
Then the following assertions hold{\em :}
   \begin{enumerate}
   \item[(i)] $\dz{|C|^2} \subseteq \dz{C^2}$,
$C^2|_{\dz{|C|^2}}$ is closed and $\dz{|C|^4}$ is a
core for $C^2|_{\dz{|C|^2}}$,
   \item[(ii)] $\|Cf\|^2 \Le \|C^2 f\| \|f\|$ for all
$f \in \dz{|C|^2}$,
   \item[(iii)] $C^2$
is closed if and only if $\dz{C^2} \subseteq
\dz{|C|^2}$,
   \item[(iv)] $C$ is
paranormal if and only if $C^2$ is closed.
   \end{enumerate}
   \end{thm}
   \begin{proof}
It follows from \eqref{zal} that $\dz{|C|^4} \subseteq
\dz{C^2}$. First we show that the operator
$C^2|_{\dz{|C|^4}}$ is closable and
   \begin{align} \label{asa}
\dz{|C|^2} = \dz{\overline{C^2|_{\dz{|C|^4}}}}.
   \end{align}
Indeed, it follows from \eqref{zal} that
   \begin{align} \label{noc1}
\||C|^2|_{\dz{|C|^4}} f\| = \|C^2|_{\dz{|C|^4}} f\|
\text{ for all } f \in \dz{|C|^4}.
   \end{align}
Since $|C|^2$ is closed, the operator
$|C|^2|_{\dz{|C|^4}}$ is closable. This and
\eqref{noc1} imply that $C^2|_{\dz{|C|^4}}$ is
closable. Applying \eqref{noc1} again and the fact
that $\dz{|C|^4}$ is a core for $|C|^2$ (see e.g.,
\cite[Theorem 4.5.1]{b-s}), we get \eqref{asa}.

Continuing the proof, we deduce from the
Cauchy-Schwarz inequality that
   \begin{align}   \label{ha1}
\|Cf\|^2 = \is{|C|^2 f}{f} \Le \||C|^2 f\|
\|f\|\overset{\eqref{noc1}} = \|C^2 f\| \|f\|, \quad f
\in \dz{|C|^4}.
   \end{align}
Take $f\in \dz{|C|^2}$. Since $\dz{|C|^4}$ is a core
for $|C|^2$, there exists a sequence
$\{f_n\}_{n=1}^\infty \subseteq \dz{|C|^4}$ such that
$f_n \to f$ and $|C|^2 f_n \to |C|^2 f$ as $n\to
\infty$. Since, by \eqref{noc1}, $\|C^2 (f_n-f_m)\| =
\||C|^2 (f_n-f_m)\|$ for all $m,n \in \nbb$, we deduce
that there exists $h \in \hh$ such that
   \begin{align} \label{ha2}
\text{$C^2 f_n \to h$ as $n\to \infty$.}
   \end{align}
It follows from \eqref{ha1} that
   \begin{align*}
\|C(f_n-f_m)\|^2 \Le \|C^2 (f_n-f_m)\| \|f_n-f_m\|,
\quad m,n \in \nbb,
   \end{align*}
which implies that there exists $g\in \hh$ such that
$Cf_n \to g$ as $n\to \infty$. Since $C$ is closed and
$f_n \to f$ as $n\to \infty$, we see that $f \in
\dz{C}$ and $Cf_n \to Cf$ as $n\to \infty$. This
combined with \eqref{ha2} implies that $f\in \dz{C^2}$
(hence $\dz{|C|^2} \subseteq \dz{C^2}$) and $C^2 f_n
\to C^2f$ as $n\to \infty$. Since, by \eqref{ha1},
$\|Cf_n\|^2 \Le \|C^2 f_n\| \|f_n\|$ for all $n\in
\nbb$, we deduce that (ii) holds. Moreover, because
$f_n \to f$ and $C^2|_{\dz{|C|^4}} f_n \to C^2f$ as
$n\to \infty$, and $C^2|_{\dz{|C|^4}}$ is closable, we
conclude that $f \in
\dz{\overline{C^2|_{\dz{|C|^4}}}}$ and $C^2f =
\overline{C^2|_{\dz{|C|^4}}}f$. This means that
$C^2|_{\dz{|C|^2}} \subseteq
\overline{C^2|_{\dz{|C|^4}}}$. Now, by \eqref{asa},
$C^2|_{\dz{|C|^2}} = \overline{C^2|_{\dz{|C|^4}}}$.
This completes the proof of the assertions (i) and
(ii).

(iii) If $C^2$ is closed, then \eqref{zal} and Lemma
 \ref{lemS4.5}(iii) imply that $|C|^4 = |C^2|^2$ and
consequently $|C|^2 = |C^2|$, which yields $\dz{C^2} =
\dz{|C^2|} = \dz{|C|^2}$ (cf.\ \cite[Lemma
8.1.1]{b-s}). The reverse implication follows from
(i).

(iv) If $C$ is paranormal, then, by \cite[Proposition
6(iv)]{StSzC}, $C^2$ is closed. The reverse
implication is a direct consequence of (ii) and (iii).
   \end{proof}
   \begin{cor} \label{paran}
Let $C \in \ogr{\hh}$. If $(C^*C)^2 = C^{*2}C^2$, then
$C$ is paranormal.
   \end{cor}
Recall that the spectral radius and the norm of a
bounded paranormal operator coincide (cf.\
\cite[Theorem 1]{istr}). Note also that the converse
implication in Corollary \ref{paran} does not hold in
general. To see this consider a bounded weighted shift
with a nonconstant monotonically increasing sequence
of positive weights.
   \begin{cor}  \label{paran2}
If $C \in \ogr{\hh}$ is either compact or algebraic,
then $C$ is normal if and only if $(C^*C)^2 =
C^{*2}C^2$. In particular, if $\hh$ is finite
dimensional, then this equivalence holds for every
operator $C$ on $\hh$.
   \end{cor}
   \begin{proof}
Apply \cite[Theorem 6.5]{c-j-s} in the case of
algebraic operators, \cite[Theorem 2]{istr} in the
case of compact operators, and Corollary \ref{paran}
in both cases.
   \end{proof}
   \section{Examples}  \label{s4}
We begin by showing that the equality \eqref{zal} does
not imply the quasinormality of $C$. In view of
Corollary \ref{paran2}, such an operator cannot be
constructed in a finite dimensional space. It is worth
mentioning that a hyponormal operator $C\in \ogr{\hh}$
which satisfies the equality \eqref{zal} must be
quasinormal (cf.\ \cite[page 63]{Em}). This means that
a non-quasinormal operator $C\in \ogr{\hh}$ which
satisfies \eqref{zal} is never hyponormal (though, it
is always paranormal, cf.\ Corollary \ref{paran}).

The first counterexample we present is related to
Toeplitz operators. We refer the reader to \cite{Arv}
and \cite{con} for more information on Toeplitz
operators.
   \begin{exa} \label{prz1}
Let $S$ be the Hardy shift on the Hardy space $H^2$.
Let $T_{\varphi}$ be a Toeplitz operator on $H^2$ with
a symbol $\varphi \in L^\infty$ such that $\varphi >
0$ almost everywhere with respect to the normalized
Lebesgue measure on the unit circle and $\int_0^{2\pi}
\varphi(\E^{\I t}) \E^{\I t} \D t \neq 0$. Then
$T_{\varphi}$ is a bounded injective positive operator
and $S$ is an isometry which does not commute with
$T_{\varphi}$ (because $\is{ST_{\varphi}1}1=0$ and
$\is{T_{\varphi}S1}1\neq 0$). This implies that the
operator $C:= S T_{\varphi}^{1/2}$ is not quasinormal.
Remembering that $S^*T_{\varphi}S = T_{\varphi}$ (cf.\
\cite[Proposition 4.2.3]{Arv}), we get
   \begin{align*}
C^{*2}C^2 = T_{\varphi}^{1/2} S^* T_{\varphi}^{1/2}
S^*S T_{\varphi}^{1/2} S T_{\varphi}^{1/2} =
T_{\varphi}^{1/2} S^* T_{\varphi} S T_{\varphi}^{1/2}
= T_{\varphi}^{2} = (C^*C)^2.
   \end{align*}
   \end{exa}
Before turning to the next example, we note that a
(unilateral or bilateral) weighted shift $W$ which
satisfies the equality $(W^*W)^2 = W^{*2}W^2$ is
quasinormal. Going a step further, we can verify that
there is no bounded non-quasinormal injective weighted
shift $W$ which satisfies the equality $(W^* W)^3 =
W^{*3}W^3$ (the details are left to the reader).
Passing to more general operators, called weighted
shifts on directed trees, we are able to construct
bounded non-quasinormal injective operators which
satisfy \eqref{zal} (or even \eqref{abc}). First, we
recall necessary definitions and state an auxiliary
result which is of some independent interest in itself
(cf.\ Proposition \ref{ws}).

Suppose $\tcal=(V,E)$ is a directed tree ($V$ and $E$
are the sets of vertices and edges of $\tcal$,
respectively). If $\tcal$ has a root, we denote it by
$\koo$. Put $V^{\circ} = V \setminus \{\koo\}$ if
$\tcal$ has a root and $V^{\circ} = V$ otherwise. For
every $u \in V^{\circ}$, there exists a unique $v \in
V$, denoted by $\pa{u}$, such that $(v,u) \in E$. The
Hilbert space of square summable complex functions on
$V$ equipped with the standard inner product is
denoted by $\ell^2(V)$. For $u \in V$, we define $e_u
\in \ell^2(V)$ to be the characteristic function of
the one-point set $\{u\}$. Given a system $\lambdab =
\{\lambda_v\}_{v\in V^{\circ}}$ of complex numbers, we
define the operator $\slam$ in $\ell^2(V)$, which is
called a {\em weighted shift} on $\tcal$ with {\em
weights} $\lambdab$, as follows
   \begin{align*}
\dz{\slam} = \{f \in \ell^2(V)\colon \Lambda_{\tcal} f
\in \ell^2(V)\} \quad \text{and} \quad \slam f =
\Lambda_{\tcal} f \text{ for } f \in \dz{\slam},
   \end{align*} where
   \begin{align*}
(\Lambda_{\tcal} f)(v) =
   \begin{cases}
   \lambda_v \cdot f(\pa{v}) & \text{if } v \in
V^{\circ},
   \\ 0 & \text{otherwise},
   \end{cases}
\quad v \in V, \, f \in \cbb^V.
   \end{align*}
We refer the reader to \cite{j-j-s} for more on
weighted shifts on directed trees.

   Now, we characterize bounded weighted shifts on
directed trees satisfying the equality
$(\slam^*\slam)^n = \slam^{*n}\slam^n$ for a fixed $n
\in \zbb_+$.
   \begin{pro}\label{ws}
Let $n \in \zbb_+$. If $\slam \in \ogr{\ell^2(V)}$ is
a weighted shift on a directed tree $\tcal=(V,E)$ with
weights $\lambdab = \{\lambda_v\}_{v\in V^{\circ}}$,
then the following two conditions are equivalent{\em
:}
   \begin{enumerate}
   \item[(i)] $(\slam^*\slam)^n = \slam^{*n}\slam^n$,
   \item[(ii)] $\|\slam e_u\|^n = \|\slam^n e_u\|$ for
all $u \in V$.
   \end{enumerate}
   \end{pro}
   \begin{proof}
By the polarization identity, (i) holds if and only if
$\||\slam|^n f\|^2 = \|\slam^n f\|^2$ for all $f \in
\ell^2(V)$. Hence an application of \cite[Proposition
3.4.3(iv)]{j-j-s} and \cite[Theorem 3.2.2(ii)]{j-j-s2}
completes the proof.
   \end{proof}
The example below deals with weighted shifts on
leafless directed trees with one branching vertex of
valency $2$ (cf.\ \mbox{\cite[page 67]{j-j-s}}).
   \begin{exa} \label{prz2}
Fix $\kappa \in \zbb_+ \sqcup \{\infty\}$. Let
$\tcal_{2,\kappa} = (V_{2,\kappa}, E_{2,\kappa})$ be
the directed tree with $V_{2,\kappa}=\{-k \colon k \in
J_{\kappa}\} \sqcup \{0\} \sqcup \{(i,j)\colon i \in
J_2, j \in \nbb\}$ and
   \begin{align*}
E_{2,\kappa} = \big\{(-k,-k+1)\colon k \in
J_\kappa\big\} &\sqcup \big\{\big(0,(i,1)\big)\colon i
\in J_2\big\}
   \\
& \sqcup \big\{\big((i,j),(i,j+1)\big)\colon i\in J_2,
j \in \nbb\big\},
   \end{align*}
where $J_{\iota} = \{k \in \nbb\colon k \Le \iota\}$
for $\iota \in \zbb_+ \sqcup \{\infty\}$ (the symbol
``$\sqcup$'' denotes disjoint union of sets). Take
$\alpha_1, \alpha_2 \in \cbb \setminus \{0\}$ such
that $|\alpha_1|^2 + |\alpha_2|^2 = 1$. Let $\beta_1,
\beta_2 \in \cbb \setminus \{0\}$ be such that
$|\alpha_1\beta_1|^2 + |\alpha_2 \beta_2|^2 = 1$ and
$(1-|\beta_1|)(1-|\beta_2|) \neq 0$ (clearly, such
$\beta_1$ and $\beta_2$ exist). Define the system of
weights $\lambdab=\{\lambda_v\}_{v\in
V_{2,\kappa}^{\circ}}$ by
   \begin{align*}
\lambda_v =
   \begin{cases}
   \alpha_i & \text{if } v = (i,1) \text{ with } i\in
   J_2,
   \\
   \beta_i & \text{if } v = (i,j) \text{ with } i\in
   J_2 \text{ and } j \Ge 2,
   \\
   1 & \text{otherwise.}
   \end{cases}
   \end{align*}
Let $\slam$ be the weighted shift on
$\tcal_{2,\kappa}$ with weights $\lambdab$. By
\cite[Proposition 3.1.8]{j-j-s}, $\slam \in
\ogr{\ell^2(V_{2,\kappa})}$. It follows from
\cite[Proposition 8.1.7(ii)]{j-j-s} (applied to $u=0$)
that $\slam$ is not quasinormal. However, by
Proposition \ref{ws}, $(\slam^*\slam)^2 =
\slam^{*2}\slam^2$.
   \end{exa}
   \begin{rem} \label{construct}
It follows from Example \ref{prz2} and \cite[Lemma
4.3.1]{j-j-s2} that there exists a bounded injective
non-quasinormal composition operator $C$ on an $L^2$
space over a $\sigma$-finite measure space such that
$(C^*C)^2 = C^{*2}C^2$.
   \end{rem}
   Modifying Example \ref{prz2}, we will show that for
every integer $n \Ge 2$, there exists a
non-quasinormal weighted shift $\slam \in
\ogr{\ell^2(V_{2,0})}$ on $\tcal_{2,0}$ such that
   \begin{align} \label{abc}
(\slam^*\slam)^n & = \slam^{*n}\slam^n \text{ and }
(\slam^*\slam)^k \neq \slam^{*k}\slam^k \text{ for all
} k \in \{2,3,\ldots\} \setminus \{n\}.
   \end{align}
   \begin{exa} \label{prz3}
Take $\alpha_1, \alpha_2, \beta_1, \beta_2 \in \cbb
\setminus \{0\}$ such that $|\alpha_1|^2 +
|\alpha_2|^2 = 1$. Define the system of weights
$\lambdab=\{\lambda_v\}_{v\in V_{2,0}^{\circ}}$ by
   \begin{align*}
\lambda_v =
   \begin{cases}
\alpha_i & \text{if } v = (i,1) \text{ with } i\in
J_2,
   \\[.5ex]
\beta_i & \text{if } v = (i,j) \text{ with } i\in J_2
\text{ and } j \Ge 2.
   \end{cases}
   \end{align*}
Let $\slam$ be the weighted shift on $\tcal_{2,0}$
with weights $\lambdab$. By \cite[Proposition
3.1.8]{j-j-s}, $\slam \in \ogr{\ell^2(V_{2,0})}$. In
view of Proposition \ref{ws}, for every integer $l\Ge
2$, $(\slam^*\slam)^l = \slam^{*l}\slam^l$ if and only
if $|\alpha_1\beta_1^{l-1}|^2 + |\alpha_2
\beta_2^{l-1}|^2 = 1$. In turn, by \cite[Proposition
8.1.7(ii)]{j-j-s}, $\slam$ is quasinormal if and only
if $|\beta_1| =|\beta_2|= 1$. Set
   \begin{align*}
f(x) = \frac{\log\big(\frac{\log(2-x)}{-\log
x}\big)}{\log(\frac{x}{2-x})}, \quad x\in (0,1).
   \end{align*}
It is a simple matter to verify that $f(x)
> 0$ for all $x\in (0,1)$, and $\lim_{x\to 0+} f(x) = 0$. Hence, there exists
$\gamma_n \in (0,1)$ such that $0 < (n-1)f(\gamma_n)
\Le 1$. A standard calculation shows that the function
$g(x)=\gamma_n^{\frac{x}{n-1}} +
(2-\gamma_n)^{\frac{x}{n-1}}$ is strictly increasing
on $[(n-1)f(\gamma_n), \infty)$, and consequently on
$[1,\infty)$. Now, taking $\alpha_1 = \alpha_2 =
\big(\frac{1}{2}\big)^{\frac12}$, $\beta_1 =
\gamma_n^{\frac{1}{2(n-1)}}$ and $\beta_2 =
(2-\gamma_n)^{\frac{1}{2(n-1)}}$, we verify that
   \begin{align*}
\alpha_1^2\beta_1^{2(k-1)} + \alpha_2^2
\beta_2^{2(k-1)} = \frac 12 g(k-1) \neq \frac 12
g(n-1) = \alpha_1^2\beta_1^{2(n-1)} + \alpha_2^2
\beta_2^{2(n-1)} = 1
   \end{align*}
whenever $k, n \Ge 2$ and $k\neq n$, which means that
$\slam$ satisfies \eqref{abc} and is not quasinormal
(the latter also follows from Lemma \ref{lemS5}).
   \end{exa}
   Concluding this section, we construct a quasinormal
operator $C$ such that $C^{*n} \subsetneq (C^n)^*$ for
every $n \Ge 2$. Recall that, by Lemma \ref{lemS5},
$C^{*n}C^{n} = (C^n)^* C^{n}$ for every quasinormal
operator $C$ and for all $n\in \zbb_+$.
   \begin{exa}  \label{prz4}
Let $S$ be an isometry of multiplicity $1$ on a
complex Hilbert space $\mm$ with a normalized cyclic
vector $e_0$. Set $e_n = S^n e_0$ for $n \in \nbb$.
Then we have
   \begin{align} \label{adj}
S^{*k} e_n =
   \begin{cases}
e_{n-k} & \text{if } k\Le n,
   \\[1ex]
0 & \text{if } k> n,
   \end{cases}
\quad k,n\in \zbb_+.
   \end{align}
Put $\hh = \bigoplus_{j=0}^\infty \mm_j$ with
$\mm_j=\mm$ for all $j\in \zbb_+$. Let
$\{r_j\}_{j=0}^\infty$ be a sequence in $\rbb_+$. Set
$C=\bigoplus_{j=0}^\infty r_j S$. By Proposition
\ref{achtenZ}(iii), the operator $C$ is quasinormal.
It follows from Proposition \ref{achtenZ}(i) that
$C^n$ is densely defined and $(C^n)^* =
\bigoplus_{j=0}^\infty r_j^n S^{*n}$ for every $n\in
\zbb_+$. This altogether implies that for every
integer $n\Ge 2$,
   \begin{align}   \label{jedenw}
   \begin{aligned}
\dz{(C^n)^*} & = \bigg\{\sumo_{\hspace{-1.2ex}j \in
\zbb_+} f_j \in \hh\colon \sum_{j=0}^\infty
r_j^{2n}\|S^{*n}f_j\|^2 < \infty\bigg\},
   \\
\dz{C^{*n}} & = \bigg\{\sumo_{\hspace{-1.2ex}j \in
\zbb_+} f_j \in \hh\colon \sum_{j=0}^\infty
r_j^{2k}\|S^{*k}f_j\|^2 < \infty \text{ for } k = 1,
2, \ldots, n\bigg\}.
   \end{aligned}
   \end{align}
Assume that $\sup_{j\in \zbb_+} r_j = \infty$. Then
there exists a sequence $\{t_j\}_{j=0}^\infty
\subseteq \rbb_+$ such~ that
   \begin{align*}
\sum_{j=0}^\infty t_j^2 < \infty \quad \text{and}
\quad \sum_{j=0}^\infty t_j^2 r_j^2 = \infty.
   \end{align*}
Hence, by \eqref{adj} and \eqref{jedenw}, we see that
$\{t_j e_{n-1}\}_{j=0}^\infty \in \dz{(C^n)^*}
\setminus \dz{C^{*n}}$ for all integers $n\Ge 2$.
Since $C^{*n} \subseteq (C^n)^*$, this shows that
$C^{*n} \subsetneq (C^n)^*$ for all integers $n \Ge
2$.
   \end{exa}
\appendix
\section{Orthogonal sums}
For sake of completeness we sketch the proof of the
following facts (cf.\ \cite{StSzB}).
   \begin{pro}\label{achtenZ}
Assume that $\hh=\bigoplus_{\omega \in \varOmega}
\hh_{\omega}$ is the orthogonal sum of complex Hilbert
spaces $\hh_{\omega}$ and $C = \bigoplus_{\omega \in
\varOmega} C_{\omega}$ is the orthogonal sum of
operators $C_{\omega}$ in $\hh_{\omega}$. Then
   \begin{enumerate}
   \item[(i)]
$C^n$ is densely defined and $(C^n)^* =
\bigoplus_{\omega \in \varOmega} (C_{\omega}^n)^*$
provided $n\in \nbb$ and $\overline{\dz{C_{\omega}^n}}
= \hh_{\omega}$ for all $\omega \in \varOmega$,
   \item[(ii)] $C^*C = \bigoplus_{\omega \in \varOmega}
C_{\omega}^* C_{\omega}$ and $|C| = \bigoplus_{\omega
\in \varOmega} |C_{\omega}|$ provided $C_{\omega}$,
$\omega \in \varOmega$, are closed and densely
defined,
   \item[(iii)] $C$ is quasinormal if and only if
$C_{\omega}$ is quasinormal for every $\omega \in
\varOmega$.
   \end{enumerate}
   \end{pro}
   \begin{proof}
(i) It is well-known that (i) holds for $n=1$. Assume
that \mbox{$n\Ge 2$}. It is easily seen that
$\dz{\bigoplus_{\omega \in \varDelta} C_{\omega}^n} =
\bigoplus_{\omega \in \varDelta} \dz{C_{\omega}^n}
\subseteq \dz{C^n}$ for every finite nonempty
$\varDelta \subseteq \varOmega$. This implies that
$C^n$ is densely defined. Since $C^n \subseteq
\bigoplus_{\omega \in \varOmega} C_{\omega}^n$, we
deduce that $\bigoplus_{\omega \in \varOmega}
(C_{\omega}^n)^* \subseteq (C^n)^*$. To prove the
reverse inclusion, take $g =
\sumo_{\hspace{-1.2ex}\omega \in \varOmega} g_{\omega}
\in \dz{(C^n)^*}$ with $g_{\omega} \in \hh_{\omega}$.
Then there exists $c \in \rbb_+$ such that
   \begin{align} \label{haha}
|\is{g}{C^n f}|^2 \Le c \|f\|^2, \quad f \in \dz{C^n}.
   \end{align}
Hence, $|\is{g_{\omega}}{C_{\omega}^n f}|^2 \Le c
\|f\|^2$ for all $f \in \dz{C_{\omega}^n}$, which
implies that $g_{\omega} \in \dz{(C_{\omega}^n)^*}$
for every $\omega \in \varOmega$. Applying
\eqref{haha} again, we see that
   \begin{align*}
\Big|\Big\langle\sumo_{\hspace{-1.2ex}\omega \in
\varDelta} (C_{\omega}^n)^*g_{\omega},
f\Big\rangle\Big|^2 = |\is{g}{C^nf}|^2 \Le c \|f\|^2,
\quad f \in {\EuScript
D}\Big(\bigoplus\nolimits_{\omega \in \varDelta}
C_{\omega}^n\Big),
   \end{align*}
which implies that $\sum_{\omega \in \varDelta}
\|(C_{\omega}^n)^*g_{\omega}\|^2 \Le c$ for every
finite nonempty $\varDelta \subseteq \varOmega$. This
yields $\sum_{\omega \in \varOmega}
\|(C_{\omega}^n)^*g_{\omega}\|^2 \Le c$. Hence $g \in
\dz{\bigoplus_{\omega \in \varOmega}
(C_{\omega}^n)^*}$, which proves (i).

(ii) Clearly $C$ is closed and densely defined.
Applying (i) with $n=1$, we get $C^*C \subseteq
\bigoplus_{\omega \in \varOmega} C_{\omega}^*
C_{\omega}$. Hence, by the maximality of selfadjoint
operators, we obtain the first equality in (ii). In
view of (i), it is clear that the operator
$\bigoplus_{\omega \in \varOmega} |C_{\omega}|$ is
positive and selfadjoint, and $(\bigoplus_{\omega \in
\varOmega} |C_{\omega}|)^2 \subseteq \bigoplus_{\omega
\in \varOmega} |C_{\omega}|^2$. Using the maximality
of selfadjoint operators and the first equality in
(ii), we obtain the second one.

(iii) Suppose $C_{\omega}$ is quasinormal for every
$\omega \in \varOmega$. It follows from (ii) that
   \begin{align*}
E_{|C|}(\varDelta) = \bigoplus_{\omega \in \varOmega}
E_{|C_{\omega}|}(\varDelta), \quad \varDelta \in
\borel{\rbb_+}.
   \end{align*}
Hence, by the inequality
$\|E_{|C_{\omega}|}(\varDelta)\| \Le 1$, we have
   \begin{align*}
E_{|C|}(\varDelta) C \subseteq \bigoplus_{\omega \in
\varOmega} E_{|C_{\omega}|}(\varDelta) C_{\omega}
\subseteq \bigoplus_{\omega \in \varOmega} C_{\omega}
E_{|C_{\omega}|}(\varDelta) = C E_{|C|}(\varDelta),
\quad \varDelta \in \borel{\rbb_+},
   \end{align*}
which shows that $C$ is quasinormal. The reverse
implication is obvious because each $\hh_{\omega}$
reduces $C$.
   \end{proof}
   \subsection*{Acknowledgement} A substantial part
of this paper was written while the first and the
third authors visited Kyungpook National University
during the spring and autumn of 2013. They wish to
thank the faculty and the administration of this unit
for their warm hospitality.
   \bibliographystyle{amsalpha}
   
   \end{document}